\documentclass{birkjour}

\usepackage[utf8]{inputenc}
\usepackage[english]{babel}
\usepackage{soul}
\usepackage{amsfonts}
\usepackage{amsmath}
\usepackage{amsthm}
\usepackage{amssymb}
\usepackage{mathrsfs}
\usepackage{extarrows}

\theoremstyle{definition}
\newtheorem{defin}{Definition}[section]
\theoremstyle{plain}
\newtheorem{teor}[defin]{Theorem}
\newtheorem*{K1}{Kato's first representation theorem}
\newtheorem*{K2}{Kato's second representation theorem}
\newtheorem{lem}[defin]{Lemma}
\newtheorem{pro}[defin]{Proposition}
\newtheorem{cor}[defin]{Corollary}
\theoremstyle{definition}
\newtheorem{esm}[defin]{Example}
\newtheorem{osr}[defin]{Remark}
\numberwithin{equation}{section}

\renewcommand{\O}{\Omega}
\newcommand{\D}{\mathcal{D}}
\renewcommand{\H}{\mathcal{H}}

\newcommand{\Set}{\mathcal{S}}
\newcommand{\B}{\mathcal{B}}
\newcommand{\Ee}{\mathcal{E}}
\newcommand{\Eo}{\mathcal{E}_\Omega}
\newcommand{\rn}{\mathfrak{n}}

\newcommand{\Up}{\Upsilon}
\renewcommand{\L}{\Lambda}
\newcommand{\Po}{\mathfrak{P}}
\newcommand{\n}[1]{\|#1\|}
\newcommand{\nor}{\|\cdot\|}

\newcommand{\noo}{\|\cdot\|_\O}

\renewcommand{\l}{\langle}
\renewcommand{\r}{\rangle}
\newcommand{\N}{\mathbb{N}}
\newcommand{\R}{\mathbb{R}}
\newcommand{\C}{\mathbb{C}}

\newcommand{\pint}{\l\cdot,\cdot\r}
\newcommand{\pin}[2]{\l#1 , #2\r}
\newcommand{\no}{\noindent}
\newcommand{\ol}{\overline}
\newcommand{\ull}{\underline} 

\newcommand{\Ma}{\mathscr{M}_{\ull{\alpha}}}
\newcommand{\Oa}{\O_{\ull{\alpha}}}

\newcommand{\sub}{\subseteq}

\newcommand{\mez}{\frac{1}{2}}
\renewcommand{\ll}{{\it l}}

\begin{document}

\title[A Survey on Solvable Sesquilinear Forms]{A Survey on Solvable Sesquilinear Forms}

\author[R. Corso]{Rosario Corso}

\address{%
	Dipartimento di Matematica e Informatica\\
	Università degli Studi di Palermo\\
	Via Archirafi 34\\
	I-90123 Palermo\\
	Italy}

\email{rosario.corso@studium.unict.it}

\begin{abstract}
	The aim of this paper is to present a unified theory of many Kato type representation theorems in terms of solvable forms on Hilbert spaces. In particular, for some sesquilinear forms $\O$ on a dense domain $\D$ one looks for an expression
	$$
	\O(\xi,\eta)=\pin{T\xi}{\eta}, \qquad \forall \xi\in D(T),\eta \in \D,
	$$
	where $T$ is a densely defined closed operator with domain $D(T)\subseteq \D$. \\
	There are two characteristic aspects of solvable forms. Namely, one is that the domain of the form can be turned into a reflexive Banach space need not be a Hilbert space. The second one is the existence of a perturbation with a bounded form which is not necessarily a multiple of the inner product.
\end{abstract}

\keywords{Kato's representation theorems, q-closed/solvable sesquilinear forms.}

\subjclass{Primary 47A07; Secondary 47A10, 47A12}

% 47B25  	Symmetric and selfadjoint operators (unbounded)
% 47A07  	Forms (bilinear, sesquilinear, multilinear)
% 47A10  	Spectrum, resolvent
% 47A12  	Numerical range, numerical radius
% 47A30  	Norms (inequalities, more than one norm, etc.)
% 47A67  	Representation theory
% 47B44  	Accretive operators, dissipative operators, etc.
% 47B50 (1973-now) Operators on spaces with an indefinite metric
% 46C50 (1991-now) Generalizations of inner products (semi-inner products, partial inner products, etc.)

\maketitle

\section{Introduction}

Let $\H$ be a Hilbert space with inner product $\pint$. The notions of bounded operators and bounded sesquilinear forms are closely related by the formula 
\begin{equation}
\label{bound}
\O(\xi,\eta)=\pin{T\xi}{\eta}, \qquad  \forall\xi,\eta \in \H.
\end{equation}
Expression (\ref{bound}) holds for every bounded sesquilinear form $\O$ and for some bounded operator by Riesz's classical representation theorem.\\
The situation in the unbounded case is more complicated. One of the earliest result of this topic is formulated by Kato in \cite{Kato}.
\begin{K1}
	Let $\O$ be a densely defined closed sectorial form with domain $\D\sub \H$. Then there exists a unique m-sectorial operator $T$, with domain $D(T)\sub \D$, such that
	\begin{equation}
	\label{Kato1}
	\O(\xi,\eta)=\pin{T\xi}{\eta}, \qquad \forall \xi\in D(T),\eta \in \D.
	\end{equation}
\end{K1}
\no Here there are some differences compared to the bounded case. For example, representation (\ref{Kato1}) does not necessarily hold on the whole $\D$, because in general $D(T)$ is smaller that $\D$. However, $D(T)$ is not a 'small' subspace since it is dense in $\H$. It is worth mentioning that an expression like (\ref{Kato1}) can be given for any sesquilinear form $\O$ considering the operator defined by
\begin{align}
\label{eq_dom_intro1}
D(T)=\{\xi \in \D:\exists \chi \in \H, \O(\xi,\eta)=\pin{\chi}{\eta}, \forall \eta \in \D\}
\end{align}
and $T\xi=\chi$, for all $\xi\in D(T)$ and $\chi$ as in (\ref{eq_dom_intro1}). %Then clearly we have the expression $\O(\xi,\eta)=\pin{T\xi}{\eta}$, for all $\xi\in D(T),\eta \in \D$, that we call the {\it representation} (of first type) of $\O$. 
Note that $T$, the operator {\it associated} to $\O$, is the maximal operator that satisfies such an expression. However, one usually is looking for some properties of $T$ concerning closedness or resolvent set, like in Kato's theorem. 

A bijection between densely defined closed sectorial forms and their associated operators (i.e. m-sectorial operators) is valid. But this bijection is not preserved when we consider a larger class of sesquilinear forms. Indeed, there exists many sesquilinear forms with the same associated operator (see Proposition 4.2 of \cite{GKMV}).
Although in the unbounded case the representation on the whole domain and the correspondence between forms and operators are lost we have the following strong result (see \cite{Kato}).
\begin{K2}
	\label{th2_Kato}
	Let $\O$ be a densely defined closed non-negative sesquilinear form  with domain $\D$ and $T$ be its positive self-adjoint associated operator. Then $\D=D(T^\mez)$ and 
	\begin{equation}
	\label{Kato2}
	\O(\xi,\eta)=\pin{T^\mez\xi}{T^\mez\eta}, \qquad \forall \xi,\eta \in \D.
	\end{equation}
\end{K2}

\no We stress that in (\ref{Kato2}) the representation is well-defined in $\D$, which is also the domain of a positive self-adjoint operator. Nevertheless, this last theorem does not have direct generalizations without the condition of positivity. Indeed, Example 2.11, Proposition 4.2 of \cite{GKMV} and Example 5.4 of \cite{FHdeS} show sesquilinear forms that satisfy the first type of representation but not the second type.

Kato's theorems lead to several applications; for instance, a way to define the Friedrichs extension of densely defined sectorial operators \cite[Ch VI.2.3]{Kato}, %a proof of the polar decomposition theorem \cite[Theorem VIII.32]{ReedSimon} 
a proof of Von Neumann's theorem about the operator $T^*T$ when $T$ is densely defined and closed \cite[Example VI.2.13]{Kato}
and a way to prove that some operators are m-sectorial or self-adjoint (see \cite[Ch. VI]{Kato} and also \cite{FHdeSW'} for some generalizations). 
There are cases where it is simpler to handle forms rather than operators. Indeed, the sum of two operators might be defined in a small subspace, but with closed forms one can define a special sum that has a dense domain (see \cite{ReedSimon} for the concrete example of the so-called {\it form sum} of the operators $Af=-f''$ and $\delta f=f(0)$ with $f\in C_0^\infty (\R)$).
%$-\frac{d^2}{dx^2}$
%In general, it is simpler to handle closed forms with the same domain rather than the associated operators. For example, the sum of the latter might be defined in a small subspace, but the {\it form sum} \cite[Ch VI.2.5]{Kato} of the operators (which is defined exactly with closed form) is densely defined.

Recently, the first representation theorem has been generalized in the context of {\it q-closed} and {\it solvable forms} in \cite{Tp_DB} and, successively, in \cite{RC_CT} (see Theorem \ref{th_rapp_risol} below). While the second one has been extended to solvable forms in \cite{Second} (see Theorem \ref{2_repr_th_2} below). Solvable forms constitute a unified theory of many representation theorems (for example \cite{FHdeS,GKMV,Lions,McIntosh68,McIntosh70,Schmitz,Schm}). %For an approach for sesquilinear forms defined on non-dense subspaces see \cite{Arendt}.
 
\no The new aspects of solvable forms, compared to the ones in the works mentioned above, are the following (see Definitions \ref{def_q_chiusa} and \ref{def_solv}). First, the structure of {\it reflexive Banach space} need not be a Hilbert space on the domain of the form. Second, the perturbation with a {\it bounded form} which is not necessarily a multiple of the inner product. These conditions are stressed in Example 7.3 of \cite{RC_CT} and Example \ref{esm_3} in Section \ref{sec:rep_th}, respectively.

This paper is organized as follows. 
In Section \ref{sec:rep_th} we give the definitions of solvable forms and their representation theorems. We show in Section \ref{sec:num} some properties of these forms in terms of the numerical range. Section \ref{sec:spec} provides an exposition of particular cases of solvable forms known in the literature. In the final section we discuss another representation called {\it Radon-Nikodym-like}.

%\section{Notations}
%\label{sec:not}

\section{The representation theorems}
\label{sec:rep_th}

Throughout this paper we will use the following notations: $\H$ is a Hilbert space with inner product $\pint$ and norm $\nor$; $\D$ is a dense subspace of $\H$; $D(T), R(T)$ and $\rho(T)$ are the {\it domain}, {\it range} and {\it resolvent set} of an operator $T$ on $\H$, respectively; $\B(\H)$ is the set of {\it bounded} operators defined everywhere on $\H$; $\Re B$ and $\Im B$ are the {\it real} and {\it imaginary parts} of an operator $B\in \B(\H)$, respectively; 
$$\rn_T:=\{\pin{T\xi}{\xi}:\xi\in D(T),\n{\xi}=1\}$$ is the {\it numerical range}  of $T$; $l_p$ with $p>1$ is the classic Banach space with the usual norm.
%
%\begin{itemize}
%	\item $\H$ is a Hilbert space, with inner product $\pint$ and norm $\nor$;
%	\item $\D$ is a dense subspace of $\H$;
%	\item $D(T), R(T)$ and $\rho(T)$ are the domain, range and resolvent set of an operator $T$ on $\H$, respectively;
%	\item $\B(\H)$ is the set of bounded operators defined everywhere on $\H$;
%	\item $\Re B$ and $\Im B$ are the real and imaginary parts of an operator $B\in \B(\H)$, respectively;
%	\item $\rn_T$ is the numerical range $\{\pin{T\xi}{\xi}:\xi\in D(T),\n{\xi}=1\}$ of $T$;
%	\item $l^p$, with $p>1$ are the classic Banach spaces with the usual norms.
%\end{itemize}

We will consider sesquilinear forms defined on $\D$, i.e., maps $\D\times \D\to \C$ which are linear in the first component and anti-linear in the second one.

\no  If $\O$ is a sesquilinear form defined on $\D$, then the {\it adjoint form} $\O^*$ of $\O$ is given by
$\O^*(\xi,\eta)=\ol{\O(\eta,\xi)}$, for all $\xi,\eta \in \D.$
The {\it real} and {\it imaginary} parts $\Re \O$ and $\Im \O$ are
$\Re \O=\frac{1}{2}(\O+\O^*) $ and $ \Im \O=\frac{1}{2i}(\O-\O^*)$, respectively.
The {\it numerical range}  of $\O$ is $$\rn_\O:=\{\O(\xi,\xi):\xi\in \D, \n{\xi}=1\}.$$
 $\O$ is said to be {\it symmetric} if $\O=\O^*$ (i.e., $\rn_\O \subseteq\R$) and, in particular, $\O$ is {\it non-negative} if $\rn_\O \subseteq [0,+\infty)$. We will denote by $\iota$ the sesquilinear form  $\iota(\xi,\eta)=\pin{\xi}{\eta}$,  $\xi,\eta \in \H$ and by $\vartheta$ the null form on $\H$.\\

%Finally, we will use the abbreviation 'w.r.t.' to mean 'with respect to'.

%\begin{itemize}
%	\item $\O^*$, for the adjoint form of $\O$, given by
%	$$
%	\O^*(\xi,\eta)=\ol{\O(\eta,\xi)} \qquad \forall \xi,\eta \in \D;
%	$$
%	\item $\Re \O$ and $\Im \O$, for the real and the imaginary parts, respectively,
%	$$
%	\Re \O=\frac{1}{2}(\O+\O^*) \qquad \Im \O=\frac{1}{2i}(\O-\O^*);
%	$$
%	\item $\rn_\O$, for the numerical range $\{\O(\xi,\xi):\xi\in \D, \n{\xi}=1\}$ of $\O$;
%	\item $\iota$, for the bounded sesquilinear form  $\iota(\xi,\eta)=\pin{\xi}{\eta}$,  $\xi,\eta \in \H$.
%\end{itemize}

\no The following definition of q-closed forms is taken from {\cite[Proposition 3.2]{RC_CT}}.

\begin{defin}
	\label{def_q_chiusa}
	A sesquilinear form $\O$ on $\D$ is called {\it  q-closed with respect to} a norm on $\D$ which is denoted by $\noo$ if
	\begin{enumerate}
		\item $\Eo:=\D[\noo]$ is a reflexive Banach space;
		\item the embedding  $\Eo\hookrightarrow \H$ is continuous;
		\item there exists $\beta\geq0$ such that $|\O(\xi,\eta)|\leq \beta \n{\xi}_\O \n{\eta}_\O$ for all $\xi,\eta\in \D$; i.e., $\O$ is bounded on $\Eo$.
	\end{enumerate}
	If $\Eo$ is a Hilbert space, then $\O$ is called {\it q-closed with respect to the inner product of $\Eo$}.
\end{defin}

\no Let $\O$ be a q-closed sesquilinear form on $\D$ w.r.t. $\noo$. We denote by $\Eo^\times$ the conjugate dual space of $\Eo:=\D[\noo]$ and by $\pin{\L}{\xi}$ the action of the conjugate linear functional $\L\in \Eo^\times$ on an element $\xi\in \D$. The reason why we use also here the symbol $\pint$ is that $\H$ is continuously embedded into $\Eo^\times$ and the action of elements of $\Eo^\times$ is an extension of the inner product of $\H$ (see \cite[Sect. 4]{RC_CT}). \\
Let $\Po(\O)$ be the set of bounded sesquilinear forms $\Up$ on $\H$ such that
\begin{enumerate}
	\item if $(\O+\Up)(\xi,\eta)=0$ for all $\eta \in \D$, then $\xi=0$;
	\item for all $\L\in\Eo^\times$ there exists $\xi\in \D$ such that 
	$\pin{\L}{\eta}=(\O+\Up)(\xi,\eta)$ for all $\eta \in \D$.
\end{enumerate}

\begin{defin}
	\label{def_solv}
	If the set $\Po(\O)$ is not empty, then $\O$ is said to be  {\it solvable w.r.t.}  $\noo$  (if moreover $\noo$ is a Hilbert norm, then $\O$ is also said to be {\it solvable w.r.t. the inner product} induced by $\noo$).
\end{defin}

\no Solvable forms are q-closed forms characterized by the existence of a bounded sesquilinear form $\Up$ on $\H$ such that the operator $X_\Up: \Eo \to \Eo^\times$ is bijective, where $
\pin{X_\Up\xi}{\eta}=\O(\xi,\eta)+\Up(\xi,\eta)$ for all $\eta \in \Eo
$ (see \cite[Lemma 5.6]{Tp_DB}). Therefore, the set $\Po(\O)$ denotes perturbations of $\O$ with bounded forms which induce a bijection of $\Eo$ onto $\Eo^\times$. Equivalent characterizations of solvable forms are provided by \cite[Lemma 5.1]{RC_CT}.

The next theorem generalizes Kato's First representation theorem for solvable forms (for the proof see \cite[Theorem 4.6]{RC_CT} and also \cite[Theorem 2.5]{Second}).

\begin{teor}
	\label{th_rapp_risol}
	Let $\O$ be a solvable sesquilinear form on $\D$ w.r.t. a norm $\noo$. Then there exists a closed operator $T$, with dense domain $D(T)\subseteq \D$ in $\H$, such that the following statements hold.
	\begin{enumerate}
		\item $\O(\xi,\eta)=\pin{T\xi}{\eta},$ for all $\xi\in D(T),\eta \in \D$.
		\item $D(T)$ is dense in $\D[\noo]$.
		\item If $\Up \in \Po(\O)$ and $B\in \B(\H)$ is the bounded operator such that
		$\Up(\xi,\eta)=\pin{B\xi}{\eta}$  for all $\xi, \eta \in \D$,
		then $0\in \rho(T+B)$. In particular, if $\Up=-\lambda \iota$, with $\lambda \in \C$, then $\lambda \in \rho(T)$, the resolvent set of  $T$.
	\end{enumerate}
	The operator $T$ is uniquely determined by the following condition. Let $\xi,\chi\in \H$. Then $\xi\in D(T)$ and $T\xi=\chi$ if and only if $\xi\in \D$ and $\O(\xi,\eta)=\pin{\chi}{\eta}$ for all $\eta$ belonging to a dense subset of $\D[\noo]$.
%	Then there exists a closed operator $T$ with dense domain $	D(T)\subseteq \D$ in $\H$, such that
%	\begin{equation}
%	\label{eq_rapp}
%	\O(\xi,\eta)=\pin{T\xi}{\eta}, \qquad \forall \xi\in D(T),\eta \in \D.
%	\end{equation}
%	Moreover,
%	\begin{enumerate}
%		\item $D(T)$ is dense in $\Eo:=\D[\noo]$;
%		\item if $T'$ is an operator with domain $D(T')\subseteq \D$ and
%		$
%		%\label{th_rapp_T'}
%		\O(\xi,\eta)=\pin{T'\xi}{\eta}
%		$ 
%		for all $\xi\in D(T')$ and $\eta$ belonging to a dense subset of $\Eo$, then $T' \subseteq T$;
%		\item a bounded form $\Up(\cdot,\cdot)=\pin{B\cdot}{\cdot}$ belongs to $\Po(\O)$ if and only if
%		$0\in \rho(T+B)$. In particular, if $\Up=-\lambda \iota$, with $\lambda \in \C$, then $\Up\in \Po(\O)$ if and only if $\lambda \in \rho(T)$;
%		\item $T$ is the unique operator satisfying \emph{(\ref{eq_rapp})} with the property that $T+B$ has range $\H$, where $\Up(\cdot,\cdot)=\pin{B\cdot}{\cdot}$ is a form belonging to $\Po(\O)$.
%	\end{enumerate}
\end{teor}	

%\no Let $\O$ be a q-closed sesquilinear form on $\D$ w.r.t. $\noo$. Let $\Up$ be a bounded sesquilinear form on $\H$. If $\xi\in \D$, we can define the conjugate linear functional $\O_\Up^\xi$ on $\Eo$ by
%$$
%\pin{\O_\Up^\xi}{\eta}=\O(\xi,\eta)+\Up(\xi,\eta) \qquad \forall \eta \in \Eo,
%$$
%which is bounded in $\Eo$, and also the operator
%\begin{align*}
%X_\Up: \Eo &\to \Eo^\times \\
%\xi &\mapsto \O_\Up^\xi,
%\end{align*}
%is bounded; i.e., $X_\Up \in \B(\Eo,\Eo^\times)$.\\
%
%\no $\Up \in \Po(\O)$ if, and only if, $X_\Up$ is a bijection of $\Eo$ onto $\Eo^\times$ if, and only if, $X_\Up$ is invertible with bounded inverse ().

%\begin{osr}
%	The operator $T$ in  Theorem \ref{th_rapp_risol} is called the {\it operator associated} to $\O$. By the second statement, 
%	\begin{align}
%	\label{eq_dom_intro}
%	D(T)=\{\xi \in \D:\exists \chi \in \H, \O(\xi,\eta)=\pin{\chi}{\eta}, \forall \eta \in \D\}
%	\end{align}
%	and $T\xi=\chi$, for all $\xi\in D(T)$ and $\chi$ as in (\ref{eq_dom_intro}).
%\end{osr}

\no Kato's second theorem is generalized in Theorem \ref{2_repr_th_2} below for the special class of hyper-solvable forms defined in the following way (see \cite[Lemma 4.14, Theorem 4.17]{Second}).

\begin{defin}
	A solvable sesquilinear form on $\D$ with associated operator $T$ is said {\it hyper-solvable} if $\D=D(|T|^\mez)$.
\end{defin}

% \no Kato's second representation theorem is generalized as follow.

\begin{teor}%[{\cite[Theorem 4.16]{Second}}]
	\label{2_repr_th_2}
	Let $\O$ be a hyper-solvable sesquilinear form on $\D$ w.r.t. a norm $\noo$ and with associated operator $T$. Then $\D=D(|T^*|^\mez)$ and 
	$$
	\O(\xi,\eta)=\pin{U|T|^\mez \xi}{|T^*|^\mez \eta}, \qquad \forall \xi,\eta \in \D,
	$$
	$$
	\O(\xi,\eta)=\pin{|T^*|^\mez U\xi}{|T^*|^\mez \eta}, \qquad \forall \xi,\eta \in \D,
	$$
	where $T=U|T|=|T^*|U$ is the polar decomposition of $T$, and $\noo$ is equivalent to the graph norms of $|T|^\mez$ and of $|T^*|^\mez$. 
\end{teor}

\begin{osr}
	\label{rem_Po_open}
	According to Theorem \ref{th_rapp_risol} and since the resolvent set of a closed operator is open we obtain the following property: if $\O$ is a solvable sesquilinear form and $\Up \in \Po(\O)$ %, then $\lambda\in \rho(T)$, where $T$ is the operator associated to $\O$, 
	then there exists $\delta >0$ such that $(\Up + \mu\iota)\in \Po(\O)$, for all $|\mu|<\delta$.
\end{osr}

\no We mention some other features of a q-closed/solvable form $\O$ w.r.t. $\noo$:
\begin{itemize}
	\item the same property of being q-closed/solvable holds for the adjoint $\O^*$ (\cite[Theorem 4.11]{RC_CT}) (this implies also that $\Re \O$ and $\Im \O$ are q-closed forms);
	\item the operators associated to $\O$ and to $\O^*$ are the adjoint of each other (\cite[Theorem 4.11]{RC_CT});
	\item the peculiarity of symmetric solvable forms to have self-adjoint associated operators (\cite[Corollary 4.14]{RC_CT});
	\item all and only norms w.r.t. which $\O$ is q-closed/solvable are equivalent to $\noo$ (\cite[Theorems 3.8, 4.4]{RC_CT});
	\item different hyper-solvable forms have different associated operators ({\cite[Theorem 5.3]{Second}}).
\end{itemize}

\begin{osr}
	Let $\O_1,\O_2$ be two q-closed sesquilinear forms on $\D$ w.r.t. $\nor_1$ and $\nor_2$, respectively, and let $c\in \C$. Then, the two norms are equivalent by \cite[Theorem 2.5]{RC_CT} and the sesquilinear forms $c\O_1,\O_1^*,\Re\O_1, \Im \O_1, \O_1+\O_2$ are q-closed w.r.t. both $\nor_1$ and $\nor_2$.
\end{osr}

\no  We conclude this section presenting some examples of solvable forms (cf. \cite[Example 4.16]{RC_CT} and \cite[Example 4.5]{Second}).

% Multiplication on l_2
\begin{esm}
	\label{esm_1}
	Let $\ull{\alpha}:=\{\alpha_n\}$ be a sequence of complex numbers and 
	$$
	\Oa(\{\xi_n\},\{\eta_n\})=\sum_{n=1}^\infty \alpha_n \xi_n \ol{\eta_n}
	$$
	with domain $\D=\{\{\xi_n\}\in \ll_2:\sum_{n=1}^\infty |\alpha_n||\xi_n|^2< \infty\}$.
	The  form $\Oa$ is hyper-solvable w.r.t. the norm given by
	$$
	\n{\{\xi_n\}}_{\Oa}= \left (\sum_{n=1}^\infty |\xi_n|^2+\sum_{n=1}^\infty |\alpha_n| |\xi_n|^2\right )^\mez.
	$$
	Moreover,
	\begin{enumerate}
		\item if  $\ol{\{\alpha_n:n\in \N\}}\neq \C$, then $-\lambda \iota\in \Po(\Oa)$, for all $\lambda\notin \ol{\{\alpha_n:n\in \N\}}$;
		\item in general\footnote{The case $\ol{\{\alpha_n:n\in \N\}}= \C$ is not considered in \cite{Tp_DB,RC_CT}.}, we set 
		$\ull{\beta}=\{\beta_n\}$ the sequence such that $\beta_n=-\alpha_n+1$ if $|\alpha_n|\leq 1$, and $\beta_n=0$ if $|\alpha_n|> 1$.	Therefore, 
		the form $\O_{\ull{\beta}}$ is bounded and $0\notin \ol{\{\alpha_n+\beta_n:n\in \N\}}$. From the previous case, $\O_{\ull{\alpha}}+\O_{\ull{\beta}}$ is solvable and $\O_{\ull{\beta}}\in \Po(\Oa)$.		
		%%$\ol{\{\alpha_n:n\in \N\}}= \C$
	\end{enumerate}
	The operator associated to $\Oa$ is the multiplication operator $\Ma$ by $\ull{\alpha}$, with domain
	$$
	D(\Ma)=\left \{\{\xi_n\}\in \mathit{l}_2 : \sum_{n=1}^\infty |\alpha_n|^2|\xi_n|^2 <\infty \right \}
	$$
	and given by $\Ma \{\xi_n\}=\{\alpha_n\xi_n\},$ for every $\{\xi_n\} \in D(\Ma)$.
\end{esm}

% Multiplication on L_2

%\begin{esm}
%	\label{esm_2}
%Let $r:\C\to \C$ be a measurable function and $\O$ the sesquilinear form with domain
%$$
%\D:=\left \{f\in L^2(\C): \int_{\C}|r(z)||f(z)|^2dz< \infty \right\}
%$$
%and given by
%$	\displaystyle
%\O(f,g)=\int_\C r(z)f(z)\ol{g(z)}dz,$ for all $f,g\in \D$.\\
%$\O$ is q-closed w.r.t. the norm
%$$ \n{f}_\O=\left (\int_{\C}(1+|r(z)|)|f(z)|^2dz\right )^{\frac{1}{2}}, \qquad  f\in \D.$$
%\end{esm}

\no The next one is a new example of solvable sesquilinear form.

% Banach space
\begin{esm}
	\label{esm_3}

Let $1<p<2$ and $q$ be such that $\frac{1}{p}+\frac{1}{q}=1$. For convenience we denote by $\underline{\xi}=\{\xi_n\}$ the generic element of the space $l_r$ with $r>1$. Let moreover $\D=l_p\oplus l_q$, which is a reflexive Banach space if it is endowed with the norm $\n{(\underline{\xi},\underline{\eta})}_\D=\n{\underline{\xi}}_p+\n{\underline{\eta}}_q$ (as usual, $\nor_p$ and $\nor_q$ are the classical norms on $l_p$ and $l_q$, respectively). The Banach space $\D[\nor_\D]$ will be denoted by $\Ee$. Observe that $\Ee$ is  isomorphic to its conjugate dual space $\Ee^\times$. Indeed, we have the isomorphism (we identify $\Ee^\times$ with $l_q\oplus l_p$)
\begin{align*}
X:\;\;\; \Ee \;\;\; &\to \;\;\;\Ee^\times \\
(\underline{\xi},\underline{\eta}) &\mapsto (\underline{\eta},\underline{\xi}).
\end{align*}
%where $g\in L^q(0,1)$ is identified with a functional on $L^p(0,1)$ in the classic way, and the same for $f\in L^p(0,1)$.
The action of $X$ is given by 
\begin{equation}
\pin{X(\underline{\xi},\underline{\eta})}{(\underline{\xi'},\underline{\eta'})}=\sum_{n=1}^\infty (\eta_n\ol{\xi_n'}+\xi_n\ol{\eta_n '})
\label{actionX} 
\end{equation}
for all $(\underline{\xi},\underline{\eta}),(\underline{\xi'},\underline{\eta'})\in \Ee$. Now we define a sesquilinear form $\O$ on $\D$ exactly by (\ref{actionX}); i.e., for $(\underline{\xi},\underline{\eta}),(\underline{\xi'},\underline{\eta'})\in \Ee$,
$$
\O ((\underline{\xi},\underline{\eta}),(\underline{\xi'},\underline{\eta'}) ):=\sum_{n=1}^\infty (\eta_n\ol{\xi_n'}+\xi_n\ol{\eta_n '}).
$$
%$$
%\O \Bigl ((\{\xi_n\},\{\eta_n\}),(\{\xi_n'\},\{\eta_n'\})\Bigl ):=\sum_{n=1}^\infty (\eta_n\ol{\xi_n'}+\xi_n\ol{\eta_n '}).
%$$
$\O$ is bounded on $\Ee$. Indeed, an easy computation shows that
\begin{align*}
\Bigl |\O ((\underline{\xi},\underline{\eta}),(\underline{\xi'},\underline{\eta'}) )\Bigl | &\leq \left| \sum_{n=1}^\infty \eta_n\ol{\xi_n'} \right| + \left| \sum_{n=1}^\infty \xi_n\ol{\eta_n '} \right| \\
& \leq \n{\underline{\xi'}}_p\n{\underline{\eta}}_q+\n{\underline{\xi}}_p\n{\underline{\eta'}}_q \\
& \leq \n{(\underline{\xi},\underline{\eta})}_\D\n{(\underline{\xi'},\underline{\eta'})}_\D.
\end{align*}
Our goal is to show that $\O$ is solvable w.r.t. $\nor_\Ee$. The first thing we need is a Hilbert space in which $\Ee$ can be continuously embedded with dense range. We can make the following considerations:
\begin{itemize}
	\item $l_p$ is continuously embedded in the Hilbert space $\H_1:=l_2$ with dense range.
	\item An inner product on $l_q$ can be given by
	$$
	[\underline{\eta},\underline{\eta'}]=\sum_{n=1}^\infty 2^{-n}\eta_n\ol{\eta_n'}, \qquad  \underline{\eta},\underline{\eta'}\in l_q.
	$$ 
	In particular, it is well-defined since
	\begin{align}
	\nonumber |[\underline{\eta},\underline{\eta'}]| &\leq \sum_{n=1}^\infty 2^{-n}|\eta_n||\eta_n'| \\
	\nonumber & \leq \sum_{m=1}^\infty 2^{-m}\n{\underline{\eta}}_q\n{\underline{\eta'}}_q \\
	\label{ineq_exm} & = \n{\underline{\eta}}_q\n{\underline{\eta'}}_q.
	\end{align}
	Let $(\H_2,[\cdot,\cdot])$ be the completion of the pre-Hilbert space $(l_q,[\cdot,\cdot])$. Moreover, by (\ref{ineq_exm}) we obtain $[\underline{\eta},\underline{\eta}]^\mez \leq \n{\underline{\eta}}_q$ for every $\underline{\eta}\in l_q$. Hence, $l_q$ is continuously embedded into $\H_2$ and of course the range is dense.
	\item $\Ee$ is continuously embedded into $\H:=\H_1\oplus\H_2$ and the range is dense.
\end{itemize}

All the arguments above prove that $\O$ is q-closed w.r.t. $\nor_\D$. Moreover, by \cite[Lemma 5.6]{Tp_DB} $\O$ is solvable w.r.t. $\nor_\D$ (indeed the operator $X_\vartheta$ coincides with $X$ which is bijective).

However, $\O$ is not solvable w.r.t. any inner product. Indeed, were it so, then $\D$ would be a Hilbert space with the same topology of $\Ee$ by \cite[Theorem 3.8]{RC_CT}. The subspace $l_p\oplus \{0\}$ is closed in $\Ee$, therefore $l_p$ would be a Hilbert space with the same topology induced by $\n{\cdot}_p$. But we know that $l_p$ is not isomorphic to a Hilbert space (for example, it is a consequence of \cite{Murray}).

Moreover, this form is not hyper-solvable by \cite[Corollary 4.4]{Second}.
\end{esm}

%\begin{teor}[{\cite[Theorems 3.8, 4.4]{RC_CT}}]
%	\label{th_q_cl_sol_norm_eq}	
%	Let $\O$ be a q-closed (respectively solvable) sesquilinear form on $\D$ w.r.t. a norm $\nor_1$ and let $\nor_2$ be a norm on $\D$. Then, $\O$ is q-closed (respectively solvable) w.r.t. $\nor_2$ if, and only if, $\nor_1$ and $\nor_2$ are equivalent.
%\end{teor}
%
%\begin{teor}[{\cite[Theorem 4.11]{RC_CT}}]
%	\label{th_ris_agg}
%	If $\O$ is a q-closed sesquilinear form on $\D$ w.r.t. a norm $\noo$, then also the adjoint  $\O^*$ is q-closed w.r.t. $\noo$. Moreover $\Up\in \Po(\O)$ if, and only if, $\Up^*\in \Po(\O^*)$, and if $\O$ is solvable w.r.t.  $\noo$, and with the associated operator $T$, then $\O^*$ is solvable w.r.t. $\noo$, with associated operator $T^*$.
%\end{teor}
%
%\begin{cor}[{\cite[Corollary 4.14]{RC_CT}}]
%	\label{cor_auto<->simm}
%	The operator associated to a solvable sesquilinear form is self-adjoint if, and only if, the form is symmetric.
%\end{cor}

\section{Numerical range}
\label{sec:num}

As it is shown in this section, the numerical range of a q-closed form plays a special role on the property of being solvable.

%\begin{teor}[{\cite[Theorem 5.2]{RC_CT}}]
%	\label{crit_gener_rn}
%	Let $\O$ be a q-closed sesquilinear form on $\D$ w.r.t. a norm $\noo$ with numerical range $\rn_\O$ and let $\Up$ be a  bounded form in $\H$. Assume that $\rn_\O\cap \rn_{-\Up}=\varnothing$, where $\rn_{-\Up}$ is the numerical range of $-\Up$. Then, $\Upsilon \in \Po(\O)$ if, and only if, either the statement 1. or 2. below holds
%	\begin{enumerate}
%		\item if $\{\xi_n\}$ is a sequence in $\D$ such that $\displaystyle \sup_{\n{\eta}_\O=1} |(\O+\Up)(\xi_n,\eta)|\to 0$, then $ \n{\xi_n}_\O\to 0$.
%		\item there exists a constant $c>0$ such that
%		$$ c\n{\xi}_\O\leq \sup_{\n{\eta}_\O=1} |(\O+\Up)(\xi,\eta)| \qquad \forall \xi \in \D.$$
%	\end{enumerate}
%\end{teor}
%
%Another way to formulate statement 2 above is that the norm $$ \sup_{\n{\eta}_\O=1} |(\O+\Up)(\cdot,\eta)|$$
%on $\D$ is equivalent to $\noo$.

\begin{lem}
	\label{conn_comp}
	Let $\O$ be a solvable sesquilinear form and let $\rn_\O\neq \C$ be its numerical range. Assume that $\mathfrak{m}$ is a connected component of  $\ol{\rn_\O}^c$, the complementary set of the closure of $\rn_\O$. Then the following statements are equivalent:
	\begin{enumerate}
		\item $-\lambda \iota \in \Po(\O)$ for some $\lambda \in \mathfrak{m}$;
		\item $-\mu \iota \in \Po(\O)$ for all $\mu \in \mathfrak{m}$.
	\end{enumerate}
\end{lem}
\begin{proof}
	It is an immediate consequence of point 3 of Theorem \ref{th_rapp_risol} and of the fact that the defect numbers of the associated operator are constant on $\mathfrak{m}$  (see \cite[Theorem V.3.2]{Kato}). 
\end{proof}

\no A particular case of this result for symmetric forms is \cite[Corollary 2.8]{Second}.

Let $\O$ be a q-closed sesquilinear form on $\D$ w.r.t. a norm $\noo$ with numerical range $\rn_\Omega\neq \C$. Let $\Up$ be a bounded form such that $\rn_\O\cap\rn_{-\Up}\neq \varnothing$, where $\rn_{-\Up}$ is the numerical range of $-\Up$ (in particular, $\Up=-\lambda \iota$ with $\lambda\notin \rn_\Omega$). Theorem 5.2 of \cite{RC_CT} gives an equivalent condition for $\Up$ to be in $\Po(\O)$. Instead in Theorem 5.4 of \cite{RC_CT} only a sufficient condition is given. \\  By point 2 of Corollary 5.3 of \cite{RC_CT}, $\Up\in \Po(\O)$ if and only if the map 
$$ \xi\mapsto \sup_{\n{\eta}_\O=1} |(\O+\Up)(\xi,\eta)|$$ defines a norm on $\D$ that
is equivalent to $\noo$.
%
%\begin{teor}
%	\label{cor_crit_qc'}
%	Let $\O$ be a q-closed sesquilinear form on $\D$ w.r.t. a norm $\noo$ with numerical range $\rn_\O$ and let $\lambda\notin \rn_\O$. Then $-\lambda \iota \in \Po(\O)$ if and only if either the statement 1 or 2 below holds
%		\begin{enumerate}
%			\item if $\{\xi_n\}$ is a sequence in $\D$ such that $\displaystyle \sup_{\n{\eta}_\O=1} |(\O-\lambda \iota)(\xi_n,\eta)|\to 0$, then $ \n{\xi_n}_\O\to 0$;
%			\item there exists a constant $c>0$ such that
%			$$ c\n{\xi}_\O\leq \sup_{\n{\eta}_\O=1} |(\O-\lambda \iota)(\xi,\eta)| \qquad \forall \xi \in \D.$$
%		\end{enumerate}
%	Assume that $\O$ is symmetric. Then it is solvable if and only if either the statement 1 or 2 above holds with any $\lambda \in\C\backslash \R$.
%\end{teor}
%\begin{proof}
%	The first part is \cite[Corollary 5.3]{RC_CT}. The second one is a direct consequence of \cite[Corollary 2.8]{Second}.
%\end{proof}

%\no We can reformulate statement 2 of Theorem \ref{cor_crit_qc'} as follows: the map 
%$$ \xi\mapsto \sup_{\n{\eta}_\O=1} |(\O-\lambda \iota)(\cdot,\eta)|$$ defines a norm on $\D$ that
%is equivalent to $\noo$. A result more general than the previous one is Theorem 5.2 of \cite{RC_CT}. 

We also mention that the numerical range of the operator associated to a solvable sesquilinear form is a dense subset of the numerical range of the form (\cite[Proposition 4.13]{RC_CT}).
%\begin{cor}
%	Let $\O$ be a q-closed symmetric sesquilinear form on $\D$ w.r.t. a norm $\noo$. Then, $\O$ is solvable if, and only if, either the statement 1. or 2. below holds
%	\begin{enumerate}
%		\item there exists $\lambda \in\C\backslash \R$ such that if $\{\xi_n\}$ is a sequence in $\D$ and $$\displaystyle \sup_{\n{\eta}_\O=1} |(\O-\lambda)(\xi_n,\eta)|\to 0,$$ then $ \n{\xi_n}_\O\to 0$;
%		\item  there exists $\lambda \in\C\backslash \R$ and a constant $c>0$ such that
%		$$ c\n{\xi}_\O\leq \sup_{\n{\eta}_\O=1} |(\O-\lambda)(\xi,\eta)| \qquad \forall \xi \in \D.$$
%	\end{enumerate}
%\end{cor}
%\begin{proof}
%	It is a direct consequence of Corollaries \ref{Po_symm} and \ref{cor_crit_qc'}.
%\end{proof}

\section{Special cases}
\label{sec:spec}

Many representation theorems for sesquilinear forms in the literature are particular cases of Theorem \ref{th_rapp_risol}. The next list explains well this assertion.

\begin{lem}
\label{th_spec}
\no Let $\O$ be a sesquilinear form on $\D$.
\begin{enumerate}
	\item $\O$ satisfies \cite[Theorem 3.3]{FHdeS} if and only if $\O$ is solvable w.r.t. an inner product, $\rn_\O\subseteq \R$ and $-\lambda \iota \in \Po(\O)$ for some $\lambda \in \R$;
	\item $\O$ satisfies \cite[Theorem 2.3]{GKMV} if and only if $\O$ is solvable w.r.t. an inner product, $\rn_\O\subseteq \R$ and $\vartheta\in \Po(\O)$;
	\item $\O$ is symmetric and satisfies Kato's First theorem %\cite[Theorem VI.2.1]{Kato}
	if and only if $\O$ is solvable and $\rn_\O$  is contained in the half-line $[\omega , +\infty)$ for some $\omega\in \R$;
	\item $\O$ satisfies Kato's First theorem %\cite[Theorem VI.2.1]{Kato} 
	if and only if $\O$ is solvable w.r.t. an inner product, $\rn_\O$ is contained in a sector $\Set:=\{\lambda\in \C:\arg(\lambda-\gamma)\leq \theta\}$, where $\gamma\in \R$ and $0\leq \theta < \frac{\pi}{2}$, and $-\lambda \iota \in \Po(\O)$ for some $\lambda \notin S^c$;
	\item $\O$ satisfies \cite[Theorem 3.1]{McIntosh68} if and only if  $\O$ is solvable w.r.t. an inner product, $\rn_\O$ is contained in the half-plane $\{\lambda\in \C:\Re \lambda \geq 0\}$ and $-i \iota \in \Po(\O)$;
	\item $\O$ satisfies \cite[Proposition 2.1]{McIntosh70} if and only if $\O$ is solvable w.r.t. an inner product and $-\lambda \iota \in \Po(\O)$ for some $\lambda \in \C$;
	\item if $\O$ satisfies \cite[Theorem 2.3]{Schmitz}, then $\O$ is solvable w.r.t. an inner product, $\rn_\O\subseteq \R$ and  $\Po(\O)$ contains a bounded form which is not in general a multiply of the inner product $\pint$;
	\item if $\O$ satisfies \cite[Theorem 11.3]{Schm}, then $\O$ is solvable w.r.t. an inner product, $\rn_\O$ is contained in a half-plane which excludes $0$ and $\vartheta\in \Po(\O)$.
\end{enumerate}
\end{lem}
\begin{proof}
	Point 3 is proved in \cite[Proposition 2.9]{Second}. Point 4 is a consequence of \cite[Proposition 7.1]{RC_CT} and Lemma \ref{conn_comp}. The other results are contained in \cite[Section 7]{RC_CT}. Note that if $\O$ satisfies \cite[Theorem 11.3]{Schm} (see also \cite{Lions}) then $|\O(\xi,\xi)|\geq \omega \n{\xi}^2$ for all $\xi\in \D$ and some constant $\omega>0$. Therefore, $0\notin \ol{\rn_\O}$ and since $\rn_\O$ is convex it is contained in a half-plane which excludes $0$.
\end{proof}

\begin{osr}
The sesquilinear forms in Example \ref{esm_1} with $\ol{\{\alpha_n:n\in \N\}}=\C$, in Example \ref{esm_3} and in Example 7.3 of \cite{RC_CT}
%\begin{enumerate}
%	\item $\Oa$ in Example \ref{esm_1} with $\ol{\{\alpha:n\in \N\}}=\C$,
%	\item $\O$  in Example \ref{esm_3},
%	\item $\O$  in Example 7.3 of \cite{RC_CT},
%\end{enumerate}
satisfy Theorem \ref{th_rapp_risol}. But the representation theorems listed in Lemma \ref{th_spec} cannot be used for these forms. 
\end{osr}

\no About the second type representation, Theorem \ref{2_repr_th_2} generalizes also Theorem
4.2 of \cite{FHdeS}, Theorem 2.10 of \cite{GKMV} and Theorem 3.1 of \cite{Schmitz}.

\section{Radon-Nikodym-like representation}
\label{sec:RN}

Theorem 3.8 of \cite{Second} provides another representation of sesquilinear forms with weaker hypothesis. %Indeed, it is only required that a form is q-closed w.r.t. an inner product.
In particular, a sesquilinear form $\O$ on $\D$ is q-closed w.r.t. an inner product if and only if 
\begin{equation}
\label{Rad_Nik_like}
\O(\xi,\eta)=\pin{QH\xi}{H\eta}, \qquad \forall \xi,\eta \in \D,
\end{equation}
where $H$ is a positive self-adjoint operator with domain $D(H)=\D$ and $0\in \rho(H)$, and $Q\in \B(\H)$.

%\begin{teor}%[{\cite[Corollary 3.5, Theorem 3.8]{Second}}]
%	\label{car_Rad_Nik}
%	A sesquilinear form $\O$ on $\D$ is q-closed w.r.t. an inner product if and only if there exist a positive self-adjoint operator $H$, with domain $D(H)=\D$ and $0\in \rho(H)$, and $Q\in \B(\H)$ such that 
%	\begin{equation}
%	\label{Rad_Nik_like}
%	\O(\xi,\eta)=\pin{QH\xi}{H\eta}, \qquad \forall \xi,\eta \in \D.
%	\end{equation}
%	Suppose that \emph{(\ref{Rad_Nik_like})} holds. Then
%	%and that $\O$ is also solvable. Then its associated operator is $T=HQH$
%	%defined in the natural domain $D(T)=\{\xi\in \D: QH\xi \in \D\}$.
%		\begin{enumerate}
%			\item a bounded form $\Up$ with associated operator $B$ belongs to $\Po(\O)$ if and only if $Q+H^{-1}BH^{-1}$ is a bijection of $\H$;
%			\item if $\O$ is also solvable then its associated operator is $T=HQH$
%			defined on the natural domain $D(T)=\{\xi\in \D: QH\xi \in \D\}$.
%		\end{enumerate}	
%\end{teor}

\no We call an expression like (\ref{Rad_Nik_like}) a {\it Radon-Nikodym-like representation} of $\O$. It is never unique (indeed we can act on $Q,H$ by multiplying with scalars) and Lemma 3.7 of \cite{Second} gives a way to obtain this type of representation.

\begin{osr}
	\begin{enumerate}
		\item[(a)] Actually, Theorem 3.8 of \cite{Second}  is another generalization of Kato's second theorem. Indeed, (\ref{Kato2}) is equal to (\ref{Rad_Nik_like}), with $Q$ being the identity operator and $H=T^\mez$.
		\item[(b)]  If $\O$ is a closed sectorial form with vertex $0$ then a Radon-Nikodym-like representation is given by formula (3.5) of \cite[Chapter VI]{Kato}.
		\item[(c)] In \cite{GKMV,Veselic} the authors dealt with sesquilinear forms like (\ref{Rad_Nik_like}). In particular, in \cite{GKMV} $Q$ is symmetric and $0\in \rho(Q)$.
		\item[(d)] The motivation of the name 'Radon-Nikodym-like' is due to a more general context (see Theorem 3.6 and Example 6.2 of \cite{Tp_DB}).
		Previous works on Radon-Nikodym style theorems, in the non-negative case, are \cite{Seb_Tit,Tarc} concerning Lebesgue decomposition of non-negative forms (see also \cite{Simon}). We mention that Theorem 2.2 of \cite{Tarc} and Theorem 3 of \cite{Seb_Tit}, with the so-called singular part null, are Kato's second version theorems in a framework with two non-negative sesquilinear forms. However, in this paper for 'Radon-Nikodym-like representation' we mean also that $D(H)=\D$  in (\ref{Rad_Nik_like}).
	\end{enumerate}
	
\end{osr}

\no Let $\mathcal{S}$ be the family of all q-closed sesquilinear forms on $\D$ w.r.t. to some inner product and $\mathcal{F}$ be the family of all positive self-adjoint operators $H$ with $D(H)=\D$ and $0\in \rho(H)$. By \cite[Lemma 3.7]{Second} we can define a map using the Radon-Nikodym-like representation as $$\mathfrak{b}:\mathcal{S}\times \mathcal{F} \to \B(\H) \qquad\text{and} \qquad\mathfrak{b}(\O,H)=Q.$$
%\begin{align*}
%\mathfrak{b}:\mathcal{S}\times \mathcal{F}\;\; &\to \B(\H) \\
%(\O,H) \; & \mapsto \;\; Q.
%\end{align*}
For a fixed $H\in \mathcal{F}$ we can also define a map as 
$$\mathfrak{b}_H:\mathcal{S}\to \B(\H) \qquad\text{and} \qquad \mathfrak{b}_H(\O)=\mathfrak{b}(\O,H)=Q,$$
where $Q$ is the operator in (\ref{Rad_Nik_like}).
%\begin{align*}
%\mathfrak{b}_H:\mathcal{S} \;\; &\to \;\; \B(\H) \\
%\O \;\; & \mapsto \; \mathfrak{b}(\O,H).
%\end{align*}
The following proposition is an immediate consequence of \cite[Lemma 3.7, Theorem 3.8, Proposition 3.12]{Second}.

\begin{pro}
	For every $H\in \mathcal{F}$, $\mathfrak{b}_H$ establishes an isomorphism between the vector spaces $\mathcal{S}$ and $\B(\H)$. Moreover, for every $\O\in \mathcal{S}$,
	$$
	\mathfrak{b}_H(\O^*)=\mathfrak{b}_H(\O)^* \qquad \mathfrak{b}_H(\Re\O)=\Re\mathfrak{b}_H(\O) \qquad \mathfrak{b}_H(\Im\O)=\Im\mathfrak{b}_H(\O).
	$$ 
\end{pro}
%\no If a sesquilinear form $\O$ has a Radon-Nikodym-like representation (\ref{Rad_Nik_like}) then $\O^*$, $\Re \O$ and $\Im \O$ have Radon-Nikodym-like representations, respectively, 
%$$
%\O^*(\xi,\eta)=\pin{Q^*H\xi}{H\eta}, \qquad \forall \xi,\eta \in \D,
%$$
%$$
%\Re \O(\xi,\eta)=\pin{(\Re Q)H\xi}{H\eta}, \qquad \forall \xi,\eta \in \D,
%$$
%$$
%\Im \O(\xi,\eta)=\pin{(\Im Q)H\xi}{H\eta}, \qquad \forall \xi,\eta \in \D.
%$$
%If $\O'$ is a q-closed sesquilinear form w.r.t an inner product, then there exists $Q'\in \B(\H)$ such that 
%$$
%\O'(\xi,\eta)=\pin{Q'H\xi}{H\eta}, \qquad \forall \xi,\eta \in \D.
%$$
%Therefore, $\O+\O'$ has Radon-Nikodym-like representation
%$$
%(\O+\O')(\xi,\eta)=\pin{(Q+Q')H\xi}{H\eta}, \qquad \forall \xi,\eta \in \D.
%$$

\begin{osr}
	\label{rem_rn_Q_O}
Let $\rn$ be one of the following subsets $(0,+\infty),[0,+\infty)$, $\R$, $\{\lambda\in \C:\Re\lambda \geq 0\}$ or $\{\lambda\in \C:\arg(\lambda)\leq \theta\}$, with $0\leq \theta < \frac{\pi}{2}$. Clearly, $\rn_Q\subseteq \rn$ if, and only if, $\rn_\O\subseteq \rn$.
%$\O$ has numerical range in $\rn$ if, and only if, $Q$ has numerical range in $\rn$.\\	
\end{osr}

\begin{cor}
	Let $\O$ be a q-closed sesquilinear form on $\D$ with Radon-Nikodym-like representation \emph{(\ref{Rad_Nik_like})}. 
	\begin{enumerate}
		\item If $0\notin \ol{\rn_Q}$, then $\O$ is solvable and $\vartheta\in \Po(\O)$.
		\item If $\rn_{\Re Q} \subseteq [\gamma,+\infty)$, with $\gamma >0$, then $\O$ is a closed sectorial form in Kato's sense.
	\end{enumerate}	
\end{cor}
\begin{proof}
%Suppose that $Q$ is normal. 
Suppose $0\notin \ol{\rn_Q}$. Then $0\in \rho(Q)$ (see \cite[Problem 214]{Halmos}) and $\O$ is solvable with $\vartheta\in \Po(\O)$ by \cite[Theorem 3.8]{Second}.\\
%In particular, if $\rn_{\Re Q} \subseteq [\gamma,+\infty)$, with $\gamma >0$, then $\Re \O$ is a solvable form with $\vartheta\in \Po(\Re\O)$. Thus, $\Re\O$ is closed semi-bouded form in Kato's sense by Theorem \ref{th_spec}.\\ %As in the remark above, $\Re\O$ has numerical range in $(0,+\infty)$ and $-\lambda \in \Po(\Re \O)$ for Remark \ref{rem_Po_open}. 
%(in this case, $\Re\O$ is a closed semi-bounded form in Kato's sense by Theorem \ref{th_spec}). 
In particular, if $\rn_{\Re Q} \subseteq [\gamma,+\infty)$, with $\gamma >0$, then $\O$ is solvable with $\vartheta\in \Po(\Re\O)$. Moreover, taking into account that $\rn_Q$ is a bounded subset, then $\rn_Q$ is contained in a sector $\Set=\{\lambda\in \C:\arg(\lambda)\leq \theta\}$, with $0\leq \theta < \frac{\pi}{2}$. As it was mentioned in Remark \ref{rem_rn_Q_O}, $\O$ has numerical range in $\Set$, and there exists $\lambda <0$ such that $-\lambda\iota \in \Po(\Re \O)$ by Remark \ref{rem_Po_open}.  %The resolvent set $\rho(T)$ contains $0$, because $\vartheta\in \Po(\O)$; hence, there exists a negative number $\lambda$ in $\rho(T)$ and, as a consequence, $-\lambda \iota \in \Po(\O)$.
Finally, Theorem \ref{th_spec} implies that $\O$ is sectorial closed in Kato's sense.
\end{proof}

%\no Radon-Nikodym-like representations are involved in a criterion for establishing if a solvable form is hyper-solvable (Theorem 4.10 of \cite{Second}).

\subsection*{Acknowledgements}
This work was supported 
%by INDAM-GNAMPA 
by the Gruppo Nazionale per l’Analisi Matematica,
la Probabilità e le loro Applicazioni (GNAMPA) of the Istituto Nazionale di Alta Matematica
(INdAM)
(project 'Problemi spettrali e di rappresentazione in quasi *-algebre di operatori'  2017).

\end{document}